\documentclass{article} 

\usepackage{graphics}

\usepackage{algorithm}
\usepackage{algorithmic}
\usepackage{enumitem}

\usepackage{amsthm}

\newtheorem{lemma}{Lemma}
\newtheorem{corollary}{Corollary}
\newtheorem{proposition}{Proposition}

\usepackage{url}
\usepackage{amsmath}

\usepackage{amsfonts}
\usepackage{dsfont}
\usepackage{upgreek}

\usepackage{graphicx}
\usepackage{epstopdf}
\graphicspath{ {./figures/} } 

\usepackage{caption}
\usepackage{subcaption}

\usepackage[numbers]{natbib}
	
\usepackage{xspace}

\newcommand{\Sec}[1]		{Sec.\,\ref{#1}}
\newcommand{\Fig}[1]		{Fig.\,\ref{#1}}

\newcommand{\Eq}[1]			{Eq.\,\ref{#1}}

\newcommand{\ie}   			{i.e.~}
\newcommand{\eg}   			{e.g.~}

\newcommand{\Erdos}   	{Erd\"os-R\'enyi }
\newcommand{\one}       {\mathds{1}}

\newcommand{\Exp}[1]    {\mathbb{E}[#1]}
\newcommand{\Prob}[1]   {\mathbb{P}(#1)}

\newcommand{\HazMat}   {\mathcal{H}}
\newcommand{\AdjMat}   {\mathcal{A}}

\title{Tight Bounds for Influence in Diffusion Networks and Application to Bond Percolation and Epidemiology}

\newcommand{\myaffiliation}{CMLA {--} ENS Cachan, CNRS, France}

\author{
R\'emi Lemonnier$^{1,2}$
\hspace{2em}
Kevin Scaman$^1$
\hspace{2em}
Nicolas Vayatis$^1$
\\$^1$ \myaffiliation
\\$^2$ 1000Mercis, Paris, France
\\\texttt{remi.lemonnier@ens-cachan.fr}
\\\texttt{\{scaman, vayatis\}@cmla.ens-cachan.fr}
}

\begin{document}
%

\maketitle

\begin{abstract}
In this paper, we derive theoretical bounds for the long-term influence of a node in an Independent Cascade Model (ICM). We relate these bounds to the spectral radius of a particular matrix and show that the behavior is sub-critical when this spectral radius is lower than $1$. More specifically, we point out that, in general networks, the sub-critical regime behaves in $O(\sqrt{n})$ where $n$ is the size of the network, and that this upper bound is met for star-shaped networks. 
We apply our results to epidemiology and percolation on arbitrary networks, and derive a bound for the critical value beyond which a giant connected component arises. Finally, we show empirically the tightness of our bounds for a large family of networks.
\end{abstract}

\section{Introduction}

The emergence of social graphs of the World Wide Web has had a considerable effect on propagation of ideas or information. For advertisers, these new diffusion networks have become a favored vector for \emph{viral marketing} operations, that consist of advertisements that people are likely to share by themselves with their social circle, thus creating a propagation dynamics somewhat similar to the spreading of a virus in epidemiology (\cite{kirby2006connected}). Of particular interest is the problem of \emph{influence maximization}, which consists of selecting the top-k nodes of the network to infect at time $t=0$ in order to maximize in expectation the final number of infected nodes at the end of the epidemic. This problem was first formulated by Domingues and Richardson in \cite{domingos2001mining} and later expressed in \cite{Kempe:2003:MSI:956750.956769} as an NP-hard discrete optimization problem under the Independent Cascade (IC) framework, a widely-used probabilistic model for information propagation.

From an algorithmic point of view, influence maximization has been fairly well studied. Assuming the transmission probability of all edges are known, Kempe, Kleinberg and Tardos (\cite{Kempe:2003:MSI:956750.956769}) derived a greedy algorithm based on Monte-Carlo simulations that was shown to approximate the optimal solution up to a factor $1-\frac{1}{e}$, building on classical results of optimization theory. Since then, various techniques were proposed in order to significantly improve the scalability of this algorithm (\cite{chen2009efficient,chen2010scalable,goyal2011celf++,ohara2013predictive}), and also to provide an estimate of the transmission probabilities from real data (\cite{gomez2010inferring,myers2010convexity}). Recently, a series of papers (\cite{DBLP:conf/icml/Gomez-RodriguezBS11, rodriguez2012influence, DBLP:conf/nips/DuSGZ13}) introduced \emph{continuous-time} diffusion networks in which infection spreads during a time period $T$ at varying rates across the different edges. While these models provide a more accurate representation of real-world networks for finite $T$, they are equivalent to the IC model when $T \rightarrow \infty$. In this paper, will focus on such long-term behavior of the contagion.

From a theoretical point of view, little is known about the influence maximization problem under the IC model framework. The most celebrated result established by Newman (\cite{newman2002spread}) proves the equivalence between bond percolation and the \emph{Susceptible-Infected-Removed} (SIR) model in epidemiology (\cite{kermack1932contributions}) that can be identified to a special case of IC model where transmission probability are equal amongst all infectious edges.

In this paper, we propose new bounds on the influence of any set of nodes. Moreover, we prove the existence of an \emph{epidemic threshold} for a key quantity defined by the spectral radius of a given \emph{hazard matrix}. Under this threshold, the influence of \emph{any} given set of nodes in a network of size $n$ will be $O(\sqrt{n})$, while the influence of a randomly chosen set of nodes will be $O(1)$. We provide empirical evidence that these bounds are sharp for a family of graphs and sets of initial influencers and can therefore be used as what is to our knowledge the first closed-form formulas for influence estimation. We show that these results generalize bounds obtained on the SIR model by Draief, Ganesh and Massouli\'e (\cite{draief2006thresholds}) and are closely related to recent results on percolation on finite inhomogeneous random graphs (\cite{bollobas2007phase}).

The rest of the paper is organized as follows. In \Sec{sec:model}, we recall the definition of Information Cascades Model and introduce useful notations. In \Sec{sec:bounds}, we derive theoretical bounds for the influence. In \Sec{sec:percolation}, we show that our results also apply to the fields of percolation and epidemiology and generalize existing results in these fields. In \Sec{sec:examples}, we illustrate our results by applying them on simple networks and retrieving well-known results. In \Sec{sec:exps}, we perform experiments in order to show that our bounds are sharp for a family of graphs and sets of initial nodes.

\section{Information Cascades Model}\label{sec:model}
\subsection{Influence in random networks and infection dynamics}

Let $\mathcal{G} = (\mathcal{V}, \mathcal{E})$ be a directed network of $n$ nodes and $A \subset \mathcal{V}$ be a set of $n_0$ nodes that are initially \emph{contagious} (\eg aware of a piece of information, infected by a disease or adopting a product). In the sequel, we will refer to $A$ as the \emph{influencers}. The behavior of the cascade is modeled using a probabilistic framework. The influencer nodes spread the contagion through the network by means of transmission through the edges of the network. More specifically, each contagious node can infect its neighbors with a certain probability. The \emph{influence} of $A$, denoted as $\sigma(A)$, is the expected number of nodes reached by the contagion originating from $A$, \ie
\begin{equation}
\sigma(A) = \sum_{v\in\mathcal{V}} \Prob{v \mbox{ is infected by the contagion } | A}.
\end{equation}

We consider three infection dynamics that we will show in the next section to be equivalent regarding the total number of infected nodes at the end of the epidemic.

\paragraph{Discrete-Time Information Cascades [$DTIC(\mathcal{P})$]}
At time $t=0$, only the influencers are infected. Given a matrix $\mathcal{P}=(p_{ij})_{ij} \in [0,1]^{n\times n}$, each node $i$ that receives the contagion at time $t$ may transmit it at time $t+1$ along its outgoing edge $(i, j)\in\mathcal{E}$ with probability $p_{ij}$. Node $i$ cannot make any attempt to infect its neighbors in subsequent rounds. The process terminates when no more infections are possible.

\paragraph{Continuous-Time Information Cascades [$CTIC(\mathcal{F},T)$]}
At time $t=0$, only the influencers are infected. Given a matrix $\mathcal{F}=(f_{ij})_{ij}$ of non-negative integrable functions, each node $i$ that receives the contagion at time $t$ may transmit it at time $s>t$  along its outgoing edge $(i, j)\in\mathcal{E}$ with stochastic rate of occurrence $f_{ij}(s-t)$. The process terminates at a given deterministic time $T>0$. This model is much richer than Discrete-time IC, but we will focus here on its behavior when $T = \infty$.

\paragraph{Random Networks [$RN(\mathcal{P})$]} Given a matrix $\mathcal{P}=(p_{ij})_{ij} \in [0,1]^{n\times n}$, each edge $(i,j) \in \mathcal{E}$ is removed independently of the others with probability $1-p_{ij}$. A node $i \in \mathcal{V}$ is said to be \emph{infected} if $i$ is linked to at least one element of $A$ in the spanning subgraph $\mathcal{G}'= (\mathcal{V}, \mathcal{E}')$ where $\mathcal{E}' \subset \mathcal{E}$ is the set of non-removed edges.
\\
\\
\indent For any $v \in \mathcal{V}$, we will designate by \emph{influence of $v$} the influence of the set containing only $v$, \ie $\sigma(\{v\})$. We will show in Section \ref{sec:perco} that, if $\mathcal{P}$ is symmetric and $\mathcal{G}$ undirected, these three infection processes are equivalent to $\emph{bond percolation}$ and the influence of a node $v$ is also equal to the expected size of the \emph{connected component} containing $v$ in $\mathcal{G}'$. This will make our results applicable to percolation in arbitrary networks. Following the percolation literature, we will denote as \emph{sub-critical} a cascade whose influence is not proportional to the size of the network $n$.

\subsection{The hazard matrix}

In order to linearize the influence problem and derive upper bounds, we introduce the concept of \emph{hazard matrix}, which describes the behavior of the information cascade. As we will see in the following, in the case of Continuous-time Information Cascades, this matrix gives, for each edge of the network, the integral of the instantaneous rate of transmission (known as hazard function). The spectral radius of this matrix will play a key role in the influence of the cascade.

\paragraph{Definition.}
For a given graph $\mathcal{G} = (\mathcal{V}, \mathcal{E})$ and edge transmission probabilities $p_{ij}$, let $\HazMat$ be the $n\times n$ matrix, denoted as the \emph{hazard matrix}, whose coefficients are
\begin{equation}
\HazMat_{ij} = \left\{
\begin{array}{ll}
-\ln(1-p_{ij}) &\mbox{if } (i, j)\in\mathcal{E}\\
0 &\mbox{otherwise}
\end{array}\right..
\end{equation}
Next lemma shows the equivalence between the three definitions of previous section.
\begin{lemma}\label{th:equivalence}
For a given graph $\mathcal{G} = (\mathcal{V}, \mathcal{E})$, set of influencers $A$, and transmission probabilities matrix $\mathcal{P}$, the probability of each node $i$ to be infected is equal under the infection dynamics $DTIC(\mathcal{P}),CTIC(\mathcal{F,\infty})$ and $RN(\mathcal{P})$, provided that for any $(i,j) \in \mathcal{E}$, $\int_0^\infty f_{ij}(t) dt= \HazMat_{ij}$.
\end{lemma}

\paragraph{Definition.}
For a given set of influencers $A\subset \mathcal{V}$, we will denote as $\HazMat(A)$ the hazard matrix except for zeros along the columns whose indices are in $A$:
\begin{equation}
\HazMat(A)_{ij} = \one_{\{j\notin A\}} \HazMat_{ij}.
\end{equation}

We recall that for any square matrix $M$, its spectral radius $\rho(M)$ is defined by $ \rho(M)=\max_i (|\lambda_i|)$ where $\lambda_1,...,\lambda_n$ are the (possibly repeated) eigenvalues of matrix $M$.
We will also use that, when $M$ is a real square matrix with positive entries,
\begin{equation}
\rho \bigg(\frac{M+M^\top}{2}\bigg)=\sup_X \frac{X^\top MX}{X^\top X}.
\end{equation}

\paragraph{Remark.}
When the $p_{ij}$ are small, the hazard matrix is very close to the transmission matrix $\mathcal{P}$. This implies that, for low $p_{ij}$ values, the spectral radius of $\HazMat$ will be very close to that of $\mathcal{P}$. More specifically, a simple calculation holds
\begin{equation}
\rho(\mathcal{P}) \leq \rho(\HazMat) \leq \frac{-\ln(1-\|\mathcal{P}\|_\infty)}{\|\mathcal{P}\|_\infty} \rho(\mathcal{P}),
\end{equation}
where $\|\mathcal{P}\|_\infty = \max_{i,j} p_{ij}$. The relatively slow increase of $\frac{-\ln(1-x)}{x}$ for $x\rightarrow 1^-$ implies that the behavior of $\rho(\mathcal{P})$ and $\rho(\HazMat)$ will be of the same order of magnitude even for high (but lower than $1$) values of $\|\mathcal{P}\|_\infty$.


\section{Upper bounds for the influence of a set of nodes}\label{sec:bounds}

Given $A \subset \mathcal{V}$ the set of influencer nodes and $|A| = n_0 < n$, we derive here two upper bounds for the influence of $A$. The first bound (Proposition \ref{th:mainResult}) applies to any set of influencers $A$ such that $|A|=n_0$. Intuitively, this result correspond to a best-case scenario  (or a worst-case scenario, depending on the viewpoint), since we can target any set of nodes so as to maximize the resulting contagion.

\begin{proposition}\label{th:mainResult}
Define $\rho_c(A) =\rho(\frac{\HazMat(A)+\HazMat(A)^\top}{2})$. Then, for any $A$ such that $|A|=n_0 < n$, denoting by $\sigma(A)$ the expected number of nodes reached by the cascade starting from $A$:

\begin{equation}
\sigma(A) \leq n_0 + \gamma_1 (n-n_0),
\end{equation}

where $\gamma_1$ is the smallest solution in $[0, 1]$ of the following equation:
\begin{equation}
\gamma_1 - 1 + \exp \left(-\rho_c(A) \gamma_1 - \frac {\rho_c(A) n_0}{\gamma_1 (n-n_0)}\right) = 0.
\end{equation}
\end{proposition}

\begin{corollary}\label{th:simpleBounds}

Under the same assumptions: 

\begin{itemize}

\item if $\rho_c(A) < 1$,
\hspace{4em}
$\displaystyle \sigma(A) \leq n_0+\sqrt{\frac{\rho_c(A)}{1-\rho_c(A)}} \sqrt{n_0(n-n_0)}$
\item if $\rho_c(A) \geq 1$,
\hspace{1.5em}
$\displaystyle \sigma(A) \leq n - (n-n_0)\exp \left(-\rho_c(A)-\frac{2\rho_c(A)}{\sqrt{4n/n_0 - 3}-1}\right)$
\end{itemize}
In particular, when $\rho_c(A) < 1$, $\sigma(A) = O(\sqrt{n})$ and the regime is sub-critical.
\end{corollary}

The second result (Proposition \ref{th:uniformResult}) applies in the case where $A$ is drawn from a uniform distribution over the ensemble of sets of $n_0$ nodes chosen amongst $n$ (denoted as $\mathcal{P}_{n_0}(\mathcal{V})$). This result corresponds to the average-case scenario in a setting where the initial influencer nodes are not known and drawn independently of the transmissions over each edge.

\begin{proposition}\label{th:uniformResult}
Define $\rho_c =\rho(\frac{\HazMat+\HazMat^\top}{2})$. Assume the set of influencers $A$ is drawn from a uniform distribution over $\mathcal{P}_{n_0}(\mathcal{V})$. Then, denoting by $\sigma_{\mbox{uniform}}$ the expected number of nodes reached by the cascade starting from $A$:

\begin{equation}
\sigma_{\mbox{uniform}} \leq n_0 + \gamma_2 (n-n_0),
\end{equation}

where $\gamma_2$ is the unique solution in $[0, 1]$ of the following equation:
\begin{equation}
\gamma_2 - 1 + \exp \left(-\rho_c \gamma_2 - \frac{\rho_c n_0}{n-n_0} \right) = 0.
\end{equation}
\end{proposition}

\begin{corollary}\label{th:uniformsimpleBounds}

Under the same assumptions: 

\begin{itemize}

\item if $\rho_c < 1$,
\hspace{8em}
$\displaystyle \sigma_{\mbox{uniform}} \leq \frac{n_0}{1-\rho_c}$
\item if $\rho_c \geq 1$,
\hspace{3.5em}
$\displaystyle \sigma_{\mbox{uniform}} \leq n - (n-n_0)\exp \left(- \frac{\rho_c}{1-\frac{n_0}{n}} \right)$
\end{itemize}
In particular, when $\rho_c < 1$, $\sigma(A) = O(1)$ and the regime is sub-critical.
\end{corollary}

Note that, in the case of undirected networks and when $p_{ij}=p$, $\rho_c = -\ln(1-p)\rho(\AdjMat)$ where $\AdjMat$ is the adjacency matrix of the network.

\section{Application to epidemiology and percolation}\label{sec:percolation}
Building on the celebrated equivalences between the fields of percolation, epidemiology and influence maximization, we show that our results generalize existing results in these fields.

\subsection{Susceptible-Infected-Removed (SIR) model in epidemiology}
We show here that Proposition 1 further improves results on the SIR model in epidemiology. This widely used model was introduced by Kermac and McKendrick (\cite{kermack1932contributions}) in order to model the propagation of a disease in a given population. In this setting, nodes represent individuals, that can be in one of three possible states, susceptible (S), infected (I) or removed (R). At $t=0$, a subset $A$ of $n_0$ nodes is infected and the epidemic spreads according to the following evolution. Each infected node transmits the infection along its outgoing edge $(i,j) \in \mathcal{E}$ at stochastic rate of occurrence $\beta$ and is removed from the graph at stochastic rate of occurrence $\delta$. The process ends for a given $T>0$. It is straightforward that, if the removed events are not observed, this infection process is equivalent to $CTIC(\mathcal{F},T)$ where for any $(i,j) \in \mathcal{E}$,$f_{ij}(t)=\beta \exp(-\delta t)$. The hazard matrix $\mathcal{H}$ is therefore equal to  $\frac{\beta}{\delta} \AdjMat$ where $\AdjMat=\big(\one_{\{(i,j)\in \mathcal{E}\}}\big)_{ij}$ is the adjacency matrix of the underlying network. Note that, by Lemma \ref{th:equivalence}, our results can be used in order to model the total number of infected nodes in a setting where infection and recovery rates of a given node exhibit a non-exponential behavior. For instance, incubation periods for different individuals generally follow a log-normal distribution \cite{nelson2007epidemiology}, which indicates that continuous-time IC with a log-normal rate of removal might be well-suited to model some kind of infections.

It was recently shown by Draief, Ganesh and Massouli\'e (\cite{draief2006thresholds}) that, in the case of undirected networks, and if $\beta \rho(\mathcal{A})<\delta$, 
\begin{equation}
\label{eqn:Massoulie}
\sigma(A) \leq \frac{\sqrt{n n_0}}{1-\frac{\beta}{\delta} \rho(\mathcal{A})}.
\end{equation}
This result shows, that, when $\rho(\mathcal{H})=\frac{\beta}{\delta} \rho(\AdjMat)<1$, the influence of set of nodes $A$ is $O(\sqrt{n})$. We show in the next lemma that this result is a direct consequence of Corollary \ref{th:simpleBounds}: the condition $\rho_c(\mathcal{A})<1$ is weaker than $\rho(\mathcal{H})<1$ and, under these conditions, the bound of Corollary \ref{th:simpleBounds} is tighter.
\begin{lemma}\label{th:lemmaMassoulie}
For any symmetric adjacency matrix $\mathcal{A}$, initial set of influencers $A$ such that $|A|=n_0 < n$, $\delta > 0$ and $\beta  < \frac{\delta}{\rho(\mathcal{A})}$, we have simultaneously $\rho_c(A) \leq \frac{\beta}{\delta} \rho(\mathcal{A})$ and 
\begin{equation}
n_0+\sqrt{\frac{\rho_c(A)}{1-\rho_c(A)}} \sqrt{n_0(n-n_0)} \leq \frac{\sqrt{n n_0}}{1-\frac{\beta}{\delta} \rho(\mathcal{A})},
\end{equation}
where the condition $\beta  < \frac{\delta}{\rho(\mathcal{A})}$ imposes that the regime is sub-critical.
\end{lemma}
 
Moreover, these new bounds capture with more accuracy the behavior of the influence in extreme cases. In the limit $\beta \rightarrow 0$, the difference between the two bounds is significant, because Proposition \ref{th:mainResult} yields $\sigma(A) \rightarrow n_0$ whereas (\ref{eqn:Massoulie}) only ensures $\sigma(A) \leq \sqrt{n n_0}$. When $n=n_0$, Proposition \ref{th:mainResult} also ensures that $\sigma(A) = n_0$ whereas (\ref{eqn:Massoulie}) yields $ \sigma(A) \leq \frac{n_0}{1-\frac{\beta}{\delta} \rho(\mathcal{A})} $. Secondly, Proposition \ref{th:mainResult} gives also bounds in the case $\beta \rho(\mathcal{A}) \geq \delta$. Finally, Proposition \ref{th:mainResult} applies to more general cases that the classical homogeneous SIR model, and allows infection and recovery rates to vary across individuals. 

\subsection{Bond percolation} \label{sec:perco}

Given a finite undirected graph $\mathcal{G}= (\mathcal{V}, \mathcal{E})$, \emph{bond percolation} theory describes the behavior of connected clusters of the spanning subgraph of $\mathcal{G}$ obtained by retaining a subset $\mathcal{E}' \subset \mathcal{E}$ of edges of $\mathcal{G}$ according to a given distribution on $\mathcal{P}(\mathcal{E})$.When these removals occur independently along each edge with same probability $1-p$, this process is called \emph{homogeneous} percolation and is fairly well known (see e.g \cite{janson2011random}). The \emph{inhomogeneous} case, where the independent edge removal probabilities $1-p_{ij}$ vary across the edges, is more intricate and has been the subject of recent studies. In particular, results on critical probabilities and size of the giant component have been obtained by Bollobas, Janson and Riordan in \cite{bollobas2007phase}. However, these bounds hold for a particular class of asymptotic graphs (inhomogeneous random graphs) when $n \rightarrow \infty$. In the next lemma, we show that our results can be used in order to obtain bounds that hold in expectation for any fixed graph. 
\begin{lemma}
\label{th:BondPerco}
Let $\mathcal{G}= (\mathcal{V}, \mathcal{E})$ be an undirected network where each edge $(i, j) \in \mathcal{E}$ has an independent probability $1-p_{ij}$ of being removed. Then, for any $v \in \mathcal{V}$, the expected size of the connected component containing $v$ is equal to the influence of $v$ in $\mathcal{G}$ under the infection process $DTIC(\mathcal{P})$.
\end{lemma}
We now derive an upper bound for $C_1(\mathcal{G'})$, the size of the largest connected component of the spanning subgraph $\mathcal{G}'=(\mathcal{V}, \mathcal{E}')$. In the following, we will denote by $\Exp{C_1(\mathcal{G}')}$ the expected value of this random variable, given $\mathcal{P}=(p_{ij})_{ij}$.
\begin{proposition}
\label{th:sizecomponent}
Let $\mathcal{G} = (\mathcal{V}, \mathcal{E})$ be a connected undirected network where each edge $(i, j) \in \mathcal{E}$ has an independent probability $1-p_{ij}$ of being removed. The expected size of the largest connected component of the resulting subgraph $\mathcal{G}'$ is upper bounded by:
\begin{equation}
\label{eqn:sizecomp}
\Exp{C_1(\mathcal{G}')} \leq n \sqrt{\gamma_3},
\end{equation}
where $\gamma_3$ is the unique solution in $[0, 1]$ of the following equation:
\begin{equation}
\gamma_3 - 1 + \frac{n-1}{n}\exp \left(-\frac{n}{n-1}\rho(\HazMat) \gamma_3 \right) = 0.
\end{equation}
Moreover, the resulting network has a probability of being connected upper bounded by:
\begin{equation}
\label{eqn:Gconnecte}
\Prob{\mathcal{G}' \mbox{ is connected}} \leq \gamma_3.
\end{equation}
\end {proposition}

In the case $\rho(\mathcal{H})<1$, we can further simplify our bounds in the same way than for Propositions \ref{th:mainResult} and \ref{th:uniformResult}.
\begin{corollary} \label{th:simpleSize}
In the case $\rho(\mathcal{H})<1$, $\Exp{C_1(\mathcal{G}')} \leq \sqrt{\frac{n}{1-\rho(\HazMat)}}$.
\end{corollary}
Whereas our results hold for any $n \in \mathbb{N}$, classical results in percolation theory study the asymptotic behavior of sequences of graphs when $n \rightarrow \infty$. In order to further compare our results, we therefore consider sequences of spanning subgraphs $(\mathcal{G'}_n)_{n \ \in \mathbb{N}}$, obtained by removing each edge of graphs of $n$ nodes $(\mathcal{G}_n)_{n \ \in \mathbb{N}}$ with probability $1-p_{ij}^{n}$. A previous result (\cite{bollobas2007phase}, Corollary 3.2 of section 5) states that, for particular sequences known as \emph{inhomogeneous random graphs} and under a given sub-criticality condition, $C_1(\mathcal{G'}_{n})=o(n)$ \emph{asymptotically almost surely} (a.a.s.), i.e with probability going to $1$ as $n \rightarrow \infty$. Using Proposition \ref{th:sizecomponent}, we get for our part the following result:
\begin{corollary}\label{th:limsup}
Assume the sequence $\left(\mathcal{H}^n = \left(-\ln(1-p_{ij}^{n})\right)_{ij}\right)_{n \ \in \mathbb{N}}$ is such that
\begin{equation}
\label{eqn:limsupassumption}
\limsup_{n \rightarrow \infty} \rho_c (\mathcal{H}^n) < 1.
\end{equation}
Then, for any $\epsilon>0$, we have asymptotically almost surely when $n \rightarrow \infty$, 
\begin{equation}
C_1(\mathcal{G}'_{n})=o(n^{1/2 + \epsilon}).
\end{equation}
\end{corollary}
This result is to our knowledge the first to bound the expected size of the largest connected component in general arbitrary networks.
\section{Application to particular networks}\label{sec:examples}
In order to illustrate our theoretical results, we now apply our bounds to three specific networks and compare them to existing results, showing that our bounds are always of the same order than these specific results. We consider three particular networks: 1)~star-shaped networks, 2)~\Erdos networks and 3)~random graphs with an expected degree distribution. In order to simplify these problems and exploit existing theorems, we will consider in this section that $p_{ij}=p$ is fixed for each edge $(i,j) \in \mathcal{E}$. Infection dynamics thus only depend on $p$, the set of influencers $A$, and the structure of the underlying network.

\subsection{Star-shaped networks}
For a star shaped network centered around a given node $v_1$, and $A = \{v_1\}$, the exact influence is computable and writes $\sigma(\{v_1\}) = 1+p(n-1)$.
As $\HazMat(A)_{ij} = -\ln(1-p) \one_{\{i = 1, j \neq 1\}}$, the spectral radius is given by
\begin{equation}
\rho \left(\frac{\HazMat(A) + \HazMat(A)^\top}{2} \right) = \frac{-\ln(1-p)}{2} \sqrt{n-1}.
\end{equation}
Therefore, Proposition \ref{th:mainResult} states that $\sigma(\{v_1\}) \leq 1+(n-1)\gamma_1$ where $\gamma_1$ is the solution of equation 
\begin{equation}
\label{eqn:starshaped}
1-\gamma_1=\exp{\left(\left(\gamma_1 \sqrt{n-1}+\frac{1}{\gamma_1 \sqrt{n-1}}\right)\frac{\ln(1-p)}{2}\right)}.
\end{equation}
It is worth mentionning that, when $p = \frac{1}{\sqrt{n-1}}$, $\gamma_1 = \frac{1}{\sqrt{n-1}}$ is solution of (\ref{eqn:starshaped}) and therefore the bound is $\sigma(\{v_1\}) \leq 1+\sqrt{n-1}$ which is tight.
Note that, in the case of star-shaped networks, the influence does not present a critical behavior and is always linear with respect to the total number of nodes $n$.

\subsection{\Erdos networks}
For \Erdos networks $\mathcal{G}(n,p)$ ({\em i.e.} an undirected network with $n$ nodes where each couple of nodes $(i,j) \in \mathcal{V}^{2}$ belongs to $\mathcal{E}$ independently of the others with probability $p$), the exact influence of a set of nodes is not known. However, percolation theory characterizes the limit behavior of the giant connected component when $n \rightarrow \infty$. In the simplest case of \Erdos networks $\mathcal{G}(n,\frac{c}{n})$ the following result holds:
\begin{lemma}
(taken from \cite{bollobas2007phase}) For a given sequence of \Erdos networks $\mathcal{G}(n,\frac{c}{n})$, we have:
\begin{itemize}
\item if $c<1$, $C_1(\mathcal{G}(n,\frac{c}{n})) \leq \frac{3}{(1-c)^{2}} \log(n)$ a.a.s.
\item if $c>1$, $C_1(\mathcal{G}(n,\frac{c}{n})) = (1+o(1))\beta n$ a.a.s. where $\beta - 1 + \exp(-\beta c)=0$.
\end{itemize}
\end{lemma}
As previously stated, our results hold for any given graph, and not only asymptotically. However, we get an asymptotic behavior consistent with the aforementioned result. Indeed, using notations of section \ref{sec:perco}, $\HazMat^{n}_{ij} = -\ln(1-\frac{c}{n}) \one_{\{i \neq j\}}$ and $\rho (\HazMat^{n}) = -(n-1) \ln(1-\frac{c}{n})$. Using Proposition \ref{th:sizecomponent}, and noting that $\gamma_3=(1+o(1))\beta$, we get that, for any $\epsilon>0$:
\begin{itemize}
\item if $c < 1$, $C_1(\mathcal{G}(n,\frac{c}{n})) = o(n^{1/2+\epsilon})$ a.a.s.
\item if $c > 1$, $C_1(\mathcal{G}(n,\frac{c}{n})) \leq (1+o(1))\beta n^{1+\epsilon}$ a.a.s., where $\beta - 1 + \exp(-\beta c)=0$.
\end{itemize}

\subsection{Random graphs with given expected degree distribution}
In this section, we apply our bounds to random graphs whose expected degree distribution is fixed (see e.g \cite{Newman:2010:NI}, section 13.2.2). More specifically, let $w = (w_i)_{i\in\{1,\dots,n\}}$ be the expected degree of each node of the network. For a fixed $w$, let $G(w)$ be a random graph whose edges are selected independently and randomly with probability
\begin{equation}
\label{eqn:expecteddegree}
q_{ij} = \frac{\one_{\{i \neq j\}} w_i w_j}{\sum_k w_k}.
\end{equation}
For these graphs, results on the \emph{volume} of connected components (i.e the expected sum of degrees of the nodes in these components)  were derived in \cite{chung2002connected} but our work gives to our knowledge the first result on the size of the giant component. Note that \Erdos $\mathcal{G}(n,p)$ networks are a special case of (\ref{eqn:expecteddegree}) where $w_i=np$ for any $i \in \mathcal{V}$.

In order to further compare our results, we note that these graphs are also very similar to the widely used \emph{configuration model} where node degrees are fixed to a sequence $w$, the main difference being that the occupation probabilities $p_{ij}$ are in this case not independent anymore. For configuration models, a giant component exists if and only if $\sum_i{w_i^{2}}> 2 \sum_i{w_i}$ (\cite{molloy1995critical,molloy1998size}). 
In the case of graphs with given expected degree distribution, we retrieve the key role played by the ratio $\frac{\sum_i w_i^2}{\sum_i w_i}$ in our criterion of non-existence of the giant component given by $\rho \left(\frac{\HazMat + \HazMat^\top}{2} \right)<1$ where
\begin{equation}
\rho \left(\frac{\HazMat + \HazMat^\top}{2} \right) \approx \rho((q_{ij})_{ij}) \leq \frac{\sum_i w_i^2}{\sum_i w_i}.
\end{equation}
The left-hand approximation is based on $-\ln(1-x) \approx x$ and is particularly good when the $q_{ij}$ are small. This is for instance the case as soon as there exists $\alpha<1$ such that, for any $i \in \mathcal{V}$, $w_i=o(n^{\alpha})$. The right-hand side is based on the fact that the spectral radius of the matrix $\left(q_{ij} + \one_{\{i = j\}}\frac{w_i^{2}}{\sum_k w_k}\right)_{ij}$ is given by $\frac{\sum_i w_i^2}{\sum_i w_i}$.

\section{Experimental results}\label{sec:exps}

In this section, we show that the bounds given in \Sec{sec:bounds} are tight (\ie very close to empirical results in particular graphs), and are good approximations of the influence on a large set of random networks.

\Fig{fig:worstCaseBoundTight} compares experimental simulations of the influence to the bound derived in proposition \ref{th:mainResult}. The considered networks have $n=1000$ nodes and are of 6 types (see e.g \cite{Newman:2010:NI} for further details on these different networks): 1)~\Erdos networks, 2)~Preferential attachment networks, 3)~Small-world networks, 4)~Geometric random networks (\cite{penrose2003random}), 5)~2D regular grids and 6)~totally connected networks with fixed weight $b\in[0,1]$ except for the ingoing and outgoing edges of the influencer node $A = \{v_1\}$ having weight $a\in[0,1]$. The results show that the bound in proposition \ref{th:mainResult} is tight (see totally connected networks in \Fig{fig:worstCaseBoundTight}) and close to the real influence for a large class of random networks. In particular, the tightness of the bound around $\rho_c(A) = 1$ validates the behavior in $\sqrt{n}$ of the worst-case influence in the sub-critical regime.
\begin{figure}[h]
	\centering
	\begin{minipage}{.45\textwidth}
		\centering
		\includegraphics[viewport=10 3 384 295, clip=true, width=1\linewidth]{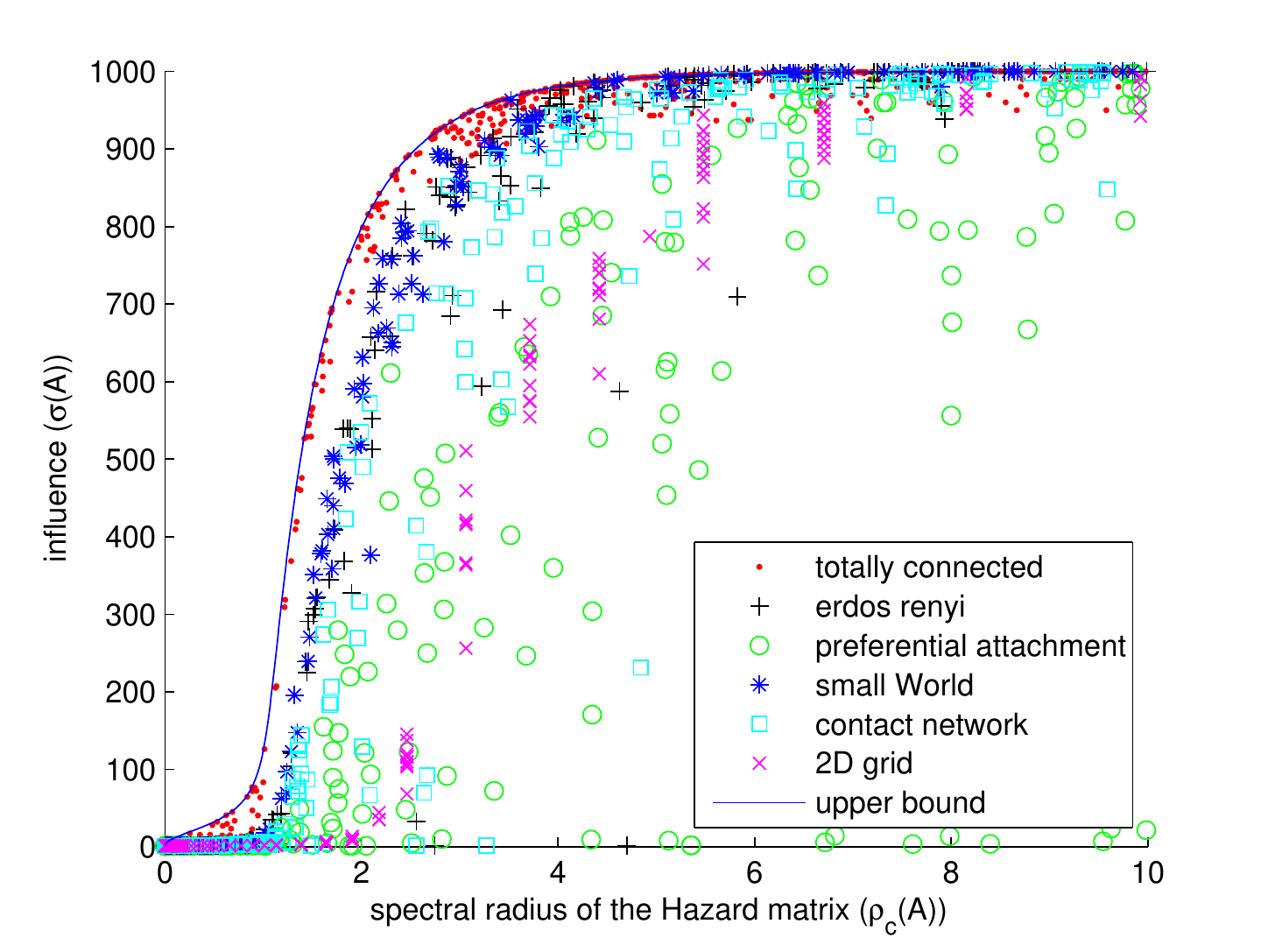}
		\caption{Empirical influence of a fixed set of influencers on random networks of various types. The solid line is the upper bound in proposition \ref{th:mainResult}.}
		\label{fig:worstCaseBoundTight}
	\end{minipage}%
	\hspace{1em}
	\begin{minipage}{.45\textwidth}
		\centering
		\includegraphics[viewport=10 3 384 295, clip=true, width=1\linewidth]{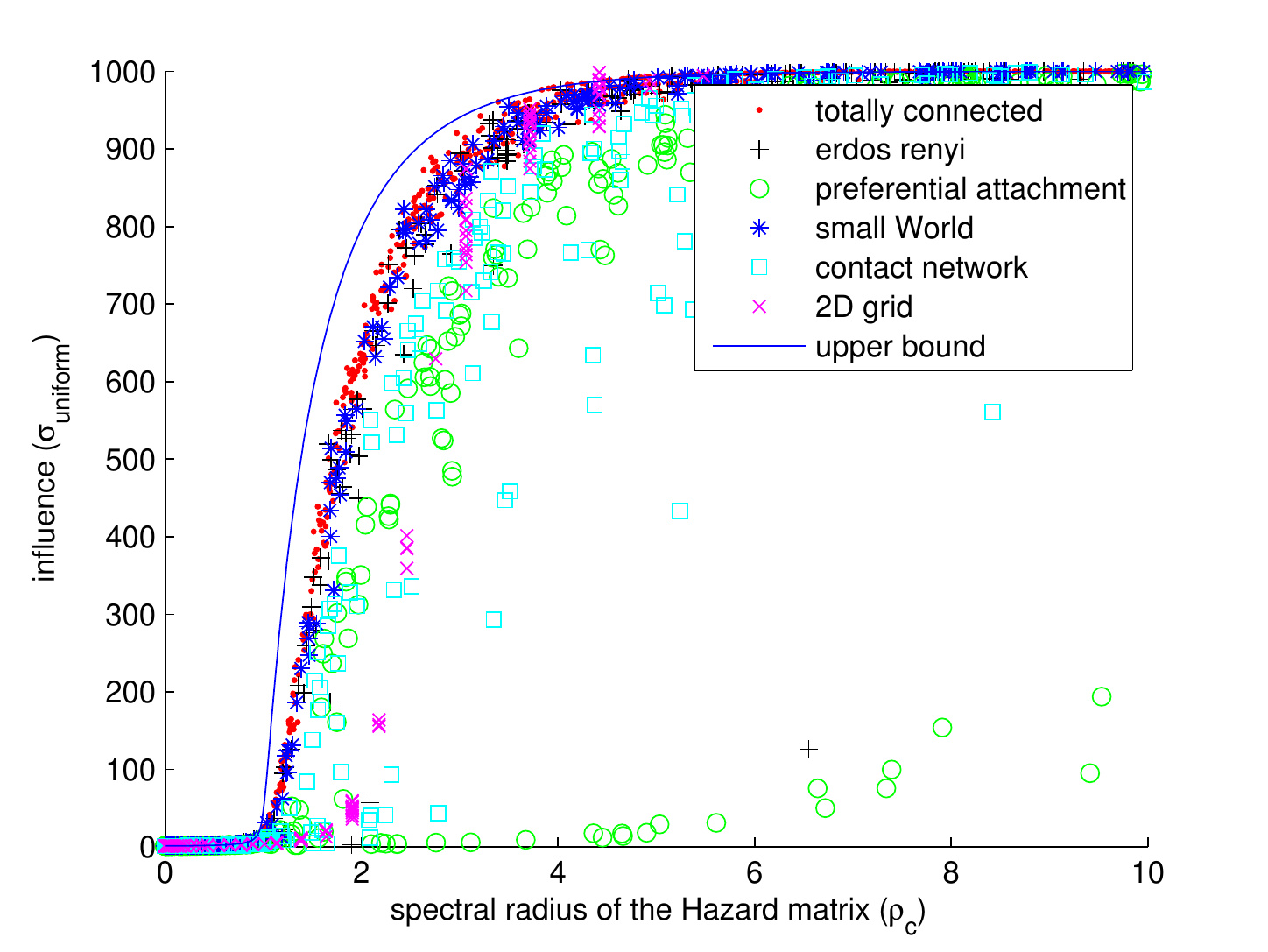}
		\caption{Empirical influence of a uniformly distributed set of influencers on random networks of various types. The solid line is the upper bound in proposition \ref{th:uniformResult}.}
		\label{fig:averageBoundTight}
	\end{minipage}
\end{figure}
\begin{figure}[h]
	\centering
	\begin{subfigure}[b]{.45\textwidth}
		\centering
		\includegraphics[viewport=10 3 384 295, clip=true, width=1\linewidth]{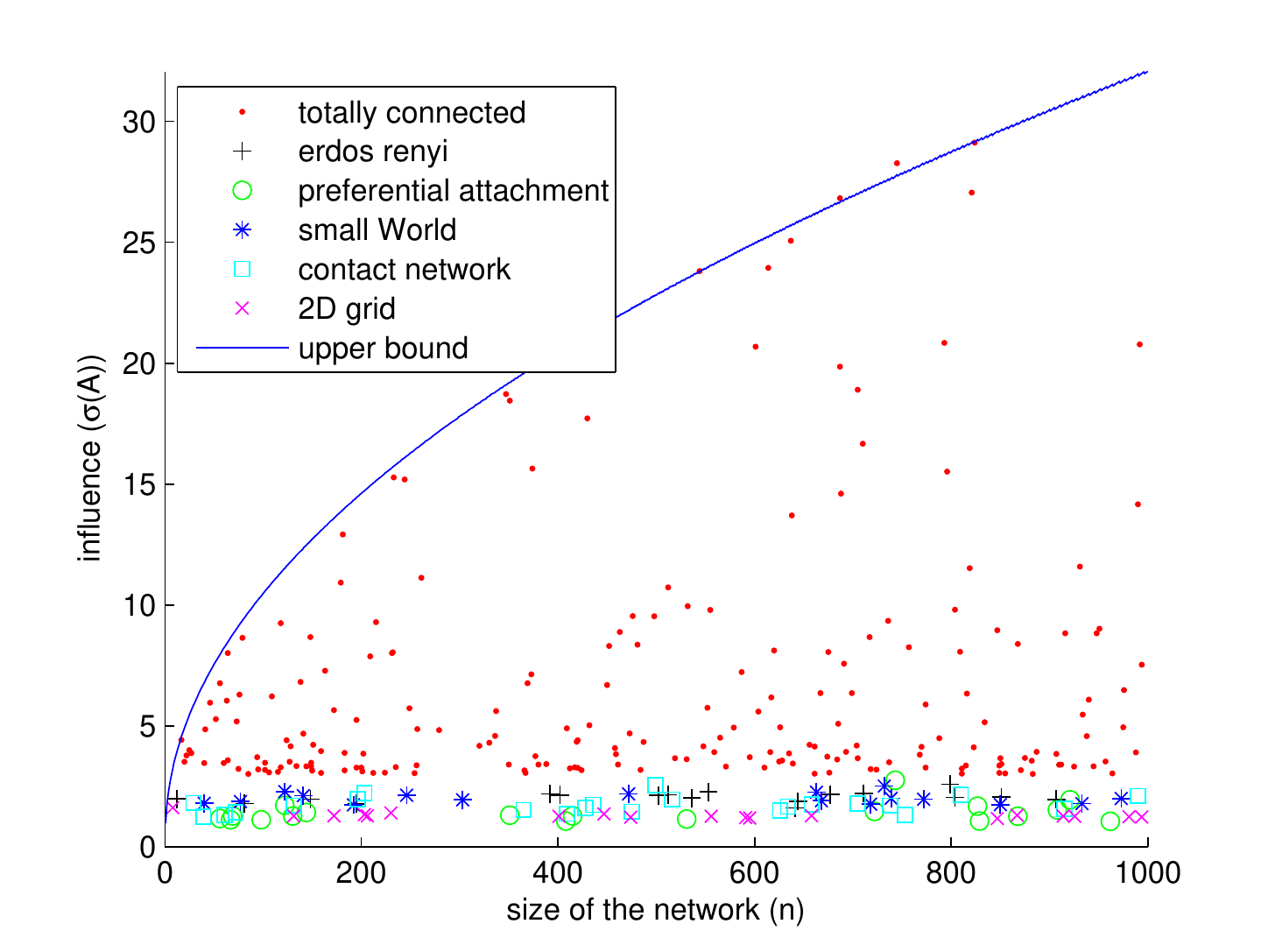}
		\caption{Sub-critical regime: $\rho_c(A) = 0.5$}
		\label{fig:rhoFixed05}
	\end{subfigure}%
	\begin{subfigure}[b]{.45\textwidth}
		\centering
		\includegraphics[viewport=10 3 384 295, clip=true, width=1\linewidth]{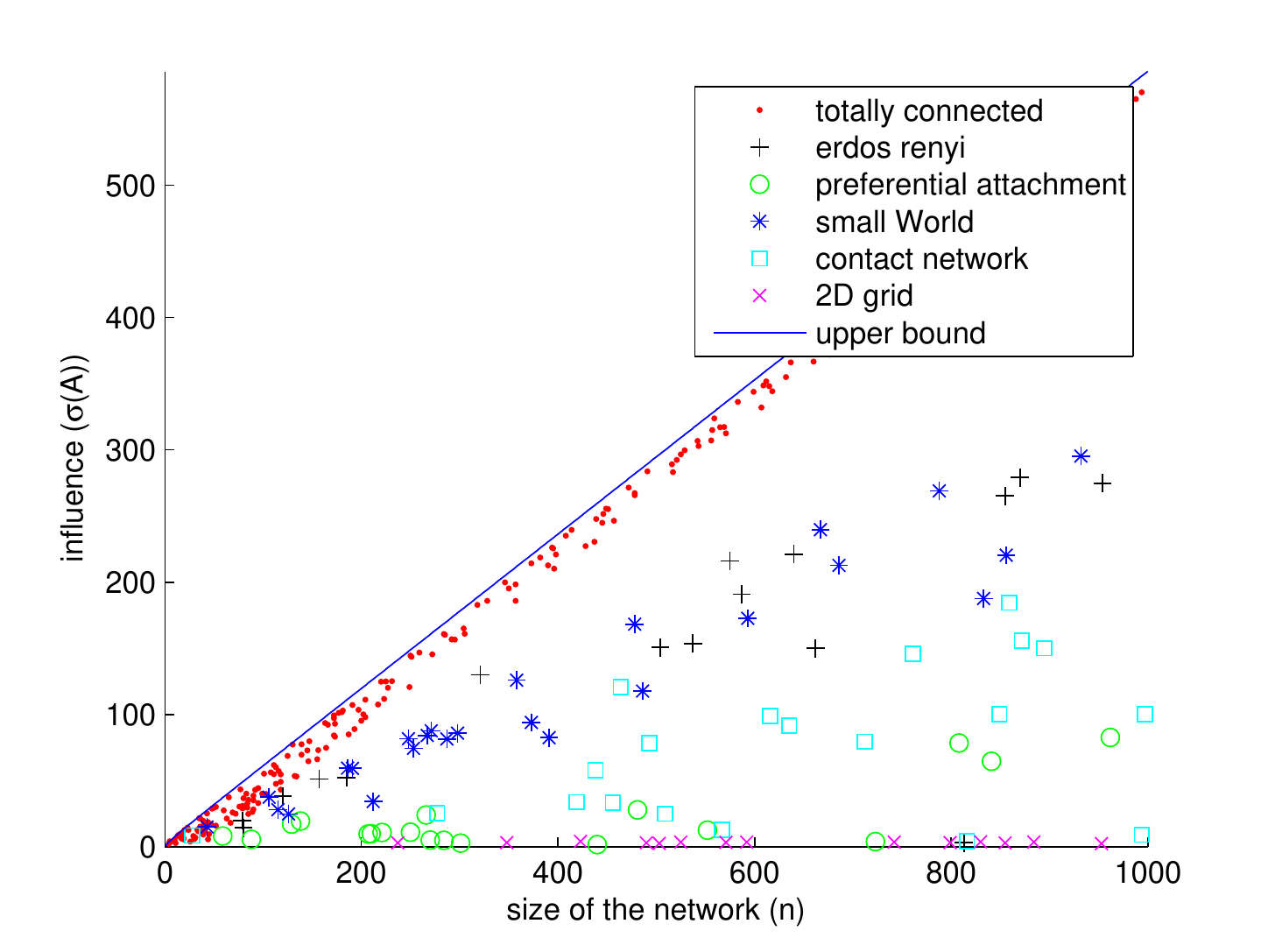}
		\caption{Super-critical regime: $\rho_c(A) = 1.5$}
		\label{fig:rhoFixed15}
	\end{subfigure}
	\caption{Influence w.r.t. the size of the network in the sub-critical ($\rho_c(A) < 1$) and super-critical regime ($\rho_c(A) > 1$). The solid line is the upper bound in proposition \ref{th:mainResult}. Note the square-root versus linear behavior of the influence.}
	\label{fig:rhoFixed}
\end{figure}

Similarly, \Fig{fig:averageBoundTight} compares experimental simulations of the influence to the bound derived in proposition \ref{th:uniformResult} in the case of random initial influencers. While this bound is not as tight as the previous one, the behavior of the bound agrees with experimental simulations, and proves a relatively good approximation of the influence under a random set of initial influencers. It is worth mentioning that the bound is tight for the sub-critical regime and shows that corollary \ref{th:uniformsimpleBounds} is a good approximation of $\sigma_{\mbox{uniform}}$ when $\rho_c < 1$.

In order to verify the criticality of $\rho_c(A) = 1$, we compared the behavior of $\sigma(A)$ w.r.t the size of the network $n$. When $\rho_c(A) < 1$ (see \Fig{fig:rhoFixed05} in which $\rho_c(A) = 0.5$), $\sigma(A) = O(\sqrt{n})$, and the bound is tight. On the contrary, when $\rho_c(A) > 1$ (see \Fig{fig:rhoFixed15} in which $\rho_c(A) = 1.5$), $\sigma(A) = O(n)$, and $\sigma(A)$ is linear w.r.t. $n$ for most random networks.

\section{Conclusion}

In this paper, we derived the first upper bounds for the influence of a given set of nodes in any finite graph under the Independent Cascade Models (ICM) framework, and relate them to the spectral radius of a given \emph{hazard matrix}. We show that these bounds can also be used to generalize previous results in the fields of epidemiology and percolation. Finally, we provide empirical evidence that these bounds are close the best possible for general graphs.


\bibliographystyle{abbrv}
\bibliography{InfluenceBounds_1}

\newpage
%
%
\section*{APPENDIX}
\subsection*{Mathematical arguments}
\subsection*{Proof of Lemma \ref{th:equivalence}}

We prove here the equivalence of propagation dynamics $DTIC(\mathcal{P}),CTIC(\mathcal{F,\infty})$ and $RN(\mathcal{P})$, provided that for any $(i,j) \in \mathcal{E}$, $\int_0^\infty f_{ij}(t) dt= \HazMat_{ij}$. More specifically, we prove the following lemma, that will be useful in the subsequent proofs. In the following, we will denote by $X_i$ the state of node $i$ at the end of the infection process, i.e $X_i=1$ if infection has reached node $i$, and $X_i=0$ otherwise. 

\begin{lemma}\label{th:exactXiInf}
Let $\mathcal{G} = (\mathcal{V}, \mathcal{E})$ be a given directed network and $A \subset \mathcal{V}$ a set of influencers. Then, under the infection processes $DTIC(\mathcal{P}),CTIC(\mathcal{F,\infty})$ and $RN(\mathcal{P})$, we have $\forall i \notin A$,
\begin{equation}
X_i = 1 - \prod_j (1 - X_j E_{ji})
\end{equation}
where the $((E_{ij})_{ij})$ are independant Bernoulli random variables $E_{ij} \sim \mathcal{B}(p_{ij})$ for infection processes $DTIC(\mathcal{P})$ and $RN(\mathcal{P})$, and $E_{ij} \sim \mathcal{B} \big(1-\exp(-\int_{0}^{\infty} f_{ij}(t)dt) \big)$ for infection process $CTIC(\mathcal{F,\infty})$.
\end{lemma}
\begin{proof}
First, note that, for $RN(\mathcal{P})$, the random variables $1_{\{(i,j) \in \mathcal{E}'\}}$ and,for $DTIC(\mathcal{P})$, the indicator function of the events that node $i$ succeeds in infecting node $j$ if $i$ is infected during the process and $j$ is still healthy at that time are independant Bernoulli variables $E_{ij} \sim \mathcal{B}(p_{ij})$ and can all be drawn at $t=0$. Moreover, by definition of the infection processes, a node is infected if and only if one of its neighbors is infected, and the respective ingoing edge transmitted the contagion. We thus have for $DTIC(\mathcal{P})$ and $RN(\mathcal{P})$:
\begin{equation}
X_i = 0 \Leftrightarrow \forall j\in\{1, \dots, n\}, X_j = 0 \mbox{ or } E_{ji} = 0,
\end{equation}
which implies that
\begin{equation}
1 - X_i = \prod_j (1 - X_j E_{ji}).
\end{equation}
For $CTIC(\mathcal{F,\infty})$, the variables drawn at the beginning of the infection process are the (possibly infinite) times $\tau_{ij}$ such that node $i$ will infect node $j$ at time $t_i+\tau_{ij}$ if node $i$ has been infected at time $t_i$, and node $j$ has not been infected by another node before time $t_i+\tau_{ij}$. By definition, these independent random variables have the following survival function:
\begin{equation}
P(\tau_{ij} < t)=1- \exp \left(-\int_{0}^{t} f_{ij}(s) ds \right)
\end{equation}
Therefore, we have by the same arguments than previously,
\begin{equation}
1 - X_i = \prod_j (1 - X_j 1_{\{\tau_{ij}<\infty\}}).
\end{equation}
which proves the result for $CTIC(\mathcal{F,\infty})$, defining  $E_{ij}=1_{\{\tau_{ij}<\infty\}}$
\end{proof}

Lemma \ref{th:equivalence}
is then a direct corollary of Lemma \ref{th:exactXiInf} in the case where, for any $(i,j) \in \mathcal{E}$, $\int_0^\infty f_{ij}(t) dt= \HazMat_{ij}$.

\subsection*{Proofs of Proposition \ref{th:mainResult} and Corollary \ref{th:simpleBounds}}

We develop here the full proofs for Proposition \ref{th:mainResult}
and Corollary \ref{th:simpleBounds}
that apply to any set of initially infected nodes. We will first need to prove two useful results: Lemma \ref{th:poscorrelated}, that proves for $j \in \mathcal{V}$ a positive correlation between the events 'node $j$ did not infect node $i$ during the epidemic' and Lemma \ref{th:mainLemma}, that bound the probability that a given node gets infected during the infection process.

\begin{lemma}\label{th:poscorrelated}
$\forall i \notin A$, $\{1 - X_j E_{ji}\}_{j\in \mathcal{V}}$ are positively correlated.
\end{lemma}

\begin{proof}
We will make use of the FKG inequality (\cite{fortuin1971correlation}):
\begin{lemma}(FKG inequality)\label{th:FKG}
Let $L$ be a finite distributive lattice, and $\mu$ a nonnegative function on $L$, such that, for any $(x,y) \in L^{2}$, 
\begin{equation}
\mu(x \vee y)\mu(x \wedge y) \leq \mu(x)\mu(y)
\end{equation}
Then, for any non-decreasing function $f$ and $g$ on $L$
\begin{equation}
\left(\sum_{x \in L} f(x)g(x)\right) \left(\sum_{x \in L}\mu(x)\right) \geq \left(\sum_{x \in L} f(x)\mu(x)\right) \left(\sum_{x \in L} g(x)\mu(x)\right)
\end{equation}
\end{lemma}
For a given set of influencers $A$, the $X_j$ are deterministic functions of the independent random variables $(E_{ij})_{ij}$. Thus, let $f_{ij}(\{E_{i'j'}\}_{(i',j')}) = 1 - X_jE_{ji}$. In order to apply the FKG inequality, we first need to show that each $f_{ij}: \{0, 1\}^{n^2} \rightarrow \{0, 1\}$ is decreasing with respect to the natural partial order on $\{0, 1\}^{n^2}$ (\ie $X \leq Y$ if $X_i \leq Y_i$ for all $i$). Let $u\in\{0, 1\}^{n^2}$ be a given transmission state of the edges of the network. In order to prove the decreasing behavior of $f_{ij}$, it is sufficient to show that $f_{ij}(u)$ is decreasing with respect to every $u_{(i,j)}$.

In order to prove this, we note that a node $i \in \mathcal{V}$ is reached by the contagion if and only if there exists a path from $A$ to $i$, such that each of its edges transmitted the contagion. This implies the following alternative expression for $X_i$:
\begin{equation}
X_i = 1 - \prod_{q\in\mathcal{Q}_i}(1 - \prod_{(j,l)\in q}E_{jl}).
\end{equation}
where $\mathcal{Q}_i$ is the collection of directed paths (without loops) in $\mathcal{G}$ from the source nodes to node $i$.

From this equation, it is obvious that $X_i(u) = 1 - \prod_{q\in\mathcal{Q}_i}(1 - \prod_{(j,l)\in q}u_{(j,l)})$ is increasing with respect to every $u_{(i,j)}$. This implies that $f_{ij}(u) = 1 - X_j(u) u_{(j,i)}$ is decreasing with respect to every $u_{(i,j)}$ and that $f_{ij}: \{0, 1\}^{n^2} \rightarrow \{0, 1\}$ is decreasing with respect to the natural partial order on $\{0, 1\}^{n^2}$.

Finally, since we consider a product measure (due to the independence of the $E_{ij}$) on a product space, we can apply the FKG inequality to $\{1 - X_j E_{ji}\}_{j\in \{1,\dots,N\}}$, and these random variables are positively correlated.
\end{proof}

The next lemma ensures that the variables $X_i$ satisfy an implicit inequation that will be the starting point of the proof of Proposition \ref{th:mainResult}.

\begin{lemma}\label{th:mainLemma}

For any $A$ such that $|A|=n_0 < n$ and for any $i \notin A$, the probability $\Exp{X_i}$  that node $i$ will be reached by the contagion originating from $A$ verifies: 

\begin{equation}
\Exp{X_i} \leq 1-\exp \bigg(- \sum_j \HazMat_{ji} \Exp{X_j} \bigg)
\end{equation}
\end{lemma}

\begin{proof}
The positive correlation of $\{1 - X_j E_{ji}\}_{j\in \{1,\dots,N\}}$ implies that
\begin{equation}
\Exp{\prod_j (1 - X_jE_{ji})} \geq \prod_j \Exp{1 - X_j E_{ji}}
\end{equation}
which leads to
\begin{equation}
\begin{array}{ll}
\Exp{X_i} &\leq 1 - \prod_j \Exp{1 - X_j E_{ji}}\\
               &= 1 - \prod_j \left(1 - \Exp{X_j}\Exp{E_{ji}}\right)\\
               &= 1 - \exp \left(\sum_j \ln(1 - \Exp{X_j}\Exp{E_{ji}})\right)\\
               &\leq 1 - \exp \left(\sum_j \ln(1 - \Exp{E_{ji}})\Exp{X_j}\right)\\
               &= 1 - \exp \left(\sum_j \HazMat_{ji} \Exp{X_j}\right)\\
\end{array}
\end{equation}
since we have on the one hand, for any $x \in [0,1]$ and $a < 1$, $\ln(1 - ax) \geq \ln(1 - a) x$, and on the other hand $\Exp{E_{ji}} = 1 - \exp(\HazMat_{ji})$ by definition of $\HazMat$.
\end{proof}

Using Lemma \ref{th:mainLemma}, we are now ready to start the proof of Proposition \ref{th:mainResult}.

\begin{proof}[Proof of Proposition \ref{th:mainResult}]

In order to simplify notations, we define $Z_i=\big(\Exp{X_i})_i$ that we collect in the vector $Z=(Z_i)_{i \in [1...n]}$. Using lemma \ref{th:mainLemma} and convexity of exponential function, we have for any $u \in R^n$ such that $\forall i\in A, u_i = 0$ and $\forall i\notin A, u_i \geq 0$,
\begin{equation}
\begin{array}{ll}
u^\top Z \leq |u|_1 \bigg (1-\sum_{i=1}^{n-1} \frac {u_i}{|u|_1} \exp(-(\HazMat^\top Z)_i) \bigg)
         \leq |u|_1 \bigg (1- \exp \big(-\frac{Z^\top \HazMat u}{|u|_1} \big) \bigg)
\end{array}
\end{equation}
where $|u|_1 = \sum_i |u_i|$ is the $L_1$-norm of $u$.

Now taking $u=(1_{i\notin A} Z_i)_i$ and noting that $\forall i \in \{1, \dots, n\}, \forall j \in A, \HazMat(A)_{ij} = 0$, we have
\begin{equation}
\begin{array}{ll}
\frac{Z^\top Z-n_0}{|Z|_1-n_0} &\leq 1- \exp \bigg(-\frac{Z^\top \HazMat(A) Z}{|Z|_1-n_0} \bigg)
\leq 1- \exp \bigg(-\frac{\rho_c(A) (Z^\top Z-n_0)}{|Z|_1-n_0}-\frac{\rho_c(A) n_0}{|Z|_1-n_0} \bigg)
\end{array}
\end{equation}
where $\rho_c(A)=\rho(\frac{\HazMat(A)+\HazMat(A)^\top}{2})$. Defining $y=\frac{Z^\top Z-n_0}{|Z|_1-n_0}$ and $z=|Z|_1-n_0=\sigma(A)-n_0$, the aforementioned inequation rewrites 
\begin{eqnarray}
y \leq 1- \exp \bigg(-\rho_c(A) y -\frac{\rho_c(A) n_0}{z} \bigg)
\end{eqnarray}
But by Cauchy-Schwarz inequality applied to $u$, $(n-n_0) (Z^\top Z-n_0) \geq (|Z|_1-n_0)^2$, which means that $z \leq y(n-n_0)$. We now consider the equation 
\begin{eqnarray}\label{eqn:equationy}
x-1+\exp \bigg (-\rho_c(A) x-\frac {\rho_c(A) n_0}{x(n-n_0)}\bigg) =0
\end{eqnarray}
Because the function $f:x \rightarrow x-1+\exp \big (-\rho_c(A) x+\frac {\rho_c(A) n_0}{x(n-n_0)}\big)$ is continuous, verifies $f(1) > 0$ and $\lim_{x \rightarrow 0^+} f(x) = -1$, equation \ref{eqn:equationy} admits a solution $\gamma_1$ in $]0,1[$.

We then prove by contradiction that $z \leq \gamma_1(n-n_0)$. Let us assume $z > \gamma_1(n-n_0)$. Then $y \leq 1- \exp \big(-\rho_c(A) y -\frac{\rho_c(A) n_0}{\gamma_1(n-n_0)} \big)$. But the function $h:x \rightarrow x-1+\exp \big (-\rho_c(A) x+\frac {\rho_c(A) n_0}{\gamma_1(n-n_0)}\big)$ is convex and verifies $h(0) < 0$ and $h(\gamma_1)=0$. Therefore, for any $y > \gamma_1$, $0=f(\gamma_1) \leq \frac{\gamma_1}{y}f(y)+(1-\frac{\gamma_1}{y})f(0)$, and therefore $f(y) > 0$. Thus, $y \leq \gamma_1$. But $z \leq y (n-n_0) \leq \gamma_1(n-n_0)$ which yields the contradiction.
\end{proof}

\begin{proof}[Proof of Corollary \ref{th:simpleBounds}]
We distinguish between the cases $\rho_c(A) > 1$ and $\rho_c(A) \leq 1$.
\paragraph{Case $\rho_c(A) < 1$.}

Using \Eq{eqn:equationy} and the fact that $\exp(z) \geq 1+z$, we get $\gamma_1 \leq \rho_c(A) \gamma_1 + \frac {\rho_c(A) n_0}{\gamma_1(n-n_0)}$ which rewrites $\gamma_1 \leq \sqrt{\frac{\rho_c(A) n_0}{(1-\rho_c(A))(n-n_0)}}$ in the case $\rho_c < 1$. Therefore,
\begin{eqnarray}
\sigma(A) \leq n_0+\sqrt{\frac{\rho_c(A)}{1-\rho_c(A)}} \sqrt{n_0(n-n_0)}
\end{eqnarray}

\paragraph{Case $\rho_c(A) \geq 1$.}

Using \Eq{eqn:equationy}, we get $\gamma_1-1+\exp(-\frac {\rho_c(A) n_0}{\gamma_1(n-n_0)}) \geq 0$, which implies $\gamma_1 \ln(\frac{1}{1-\gamma_1}) \geq \frac{\rho_c(A) n_0}{n-n_0} \geq \frac{n_0}{n-n_0}$. By concavity of the logarithm, we therefore have $\gamma_1^2 \geq \frac{n_0(1-\gamma_1)}{n-n_0}$ which means that 
$\gamma_1(n-n_0) \geq \frac{n_0(\sqrt{4n/n_0 - 3}-1)}{2}$. By plugging this lower bound in \Eq{eqn:equationy}, we obtain 

\begin{eqnarray}
\sigma(A) \leq n_0+\bigg(1-\exp \big(-\rho_c(A)-\frac{2\rho_c(A)}{\sqrt{4n/n_0-3}-1}\big)\bigg) (n-n_0)
\end{eqnarray}
\end{proof}

\subsection*{Proofs of Proposition \ref{th:uniformResult} and Corollary \ref{th:uniformsimpleBounds}}

In this subsection, we develop the proofs for Proposition \ref{th:uniformResult}
and Corollary \ref{th:uniformsimpleBounds}
in the case when the set of initially infected node is drawn from a uniform distribution over $\mathcal{P}_{n_0}(\mathcal{V})$.

We start with an important lemma that will play the same role in the proof of Proposition \ref{th:uniformResult}
than Lemma \ref{th:mainLemma} in the proof of Proposition \ref{th:mainResult}.

\begin{lemma}\label{th:mainLemmaUniform}

Define $\rho_c =\rho(\frac{\HazMat+\HazMat^\top}{2})$. Assume $A$ is drawn from an uniform distribution over $\mathcal{P}_{n_0}(\mathcal{V})$. Then, for any $i \in \mathcal{V}$, the probability $\Exp{X_i}$  that node $i$ will be reached by the contagion satisfies the following implicit inequation: 

\begin{equation}
\Exp{X_i} \leq 1 - \frac{n-n_0}{n} \exp \bigg(- \frac{n}{n-n_0} \sum_j \HazMat_{ji} \Exp{X_i} \bigg)
\end{equation}
\end{lemma}

\begin{proof}
\begin{equation}
\begin{array}{ll}
\Exp{X_i} &=\Exp{1_{\{i \in A\}}}+\Exp{1_{\{i \notin A\}}}\Exp{\Exp{X_i|A}|i \notin A} \\
            &\leq \frac{n_0}{n} + \frac{n-n_0}{n} \bigg(1 - \Exp{\exp \left(-\sum_j \HazMat_{ji} \Exp{X_j|A} \right)|i \notin A} \bigg)\\
            &\leq \frac{n_0}{n} + \frac{n-n_0}{n} \bigg(1 - \exp \left(-\Exp{\sum_j \HazMat_{ji} \Exp{X_j|A} |i \notin A}\right) \bigg)\\
            &= 1 - \frac{n-n_0}{n} \exp \left(-\sum_j \HazMat_{ji} \Exp{X_j|i \notin A} \right)\\
            &\leq 1 - \frac{n-n_0}{n} \exp \left(-\frac{n}{n-n_0} \sum_j \HazMat_{ji} \Exp{X_j} \right)
\end{array}
\end{equation}
where the first inequality is Lemma \ref{th:mainLemma} and the second one is Jensen inequality for conditional expectations.
\end{proof}
\begin{proof}[Proof of Proposition \ref{th:uniformResult}]
We define $Z_i=\big(\Exp{X_i})_i$ that we collect in the vector $Z=(Z_i)_{i \in [1...n]}$. Then, using Lemma \ref{th:mainLemmaUniform}, and convexity of exponential function, we have:
\begin{equation}
\begin{array}{ll}
\frac{Z^\top Z}{|Z|_1} &\leq \bigg (1- \frac{n-n_0}{n} \sum_{i=1}^n \frac {Z_i}{|Z|_1} \exp \big(-\frac{n}{n-n_0} (\HazMat^\top Z)_i\big) \bigg)\\
                       &\leq \bigg (1- \frac{n-n_0}{n} \exp \big(-\frac{n}{n-n_0} \frac{Z^\top \HazMat Z}{|Z|_1} \big) \bigg)
\end{array}
\end{equation}
Now, defining $y=\frac{Z^\top Z}{|Z|_1}$, we have by Cauchy-Schwarz inequality ${|Z|_1} \leq ny$ where $ y \leq 1- \frac{n-n_0}{n} \exp \big(-\frac{n}{n-n_0} \rho_c y \big)$. Because function $f:x \rightarrow x-1 + \frac{n-n_0}{n} \exp \big(-\frac{n}{n-n_0} \rho_c y \big)$ is continuous and convex over $]0,1[$, $f(0) < 0$ and $f(1) > 0$, there exists a solution $\gamma \in ]0,1[$ of the equation $f(x)=0$. By the same arguments than in proof of Proposition \ref{th:mainResult}
, we have that, for any $z \in [0,1]$, $f(z) \leq 0 \Rightarrow z \leq \gamma$. This proves the uniqueness of $\gamma$ as well as the fact that $y \leq \gamma$. Now, defining $\gamma_2 = \frac{n_0}{n}+\frac{n-n_0}{n}\gamma$, we have on the one hand \begin{equation}
\sigma_{\mbox{uniform}} \leq n_0 + \gamma_2 (n-n_0)
\end{equation}
and on the other hand
\begin{equation}
\gamma_2 - 1 + \exp \left(-\rho_c \gamma_2 - \frac{\rho_c n_0}{n-n_0} \right) = 0
\end{equation}
which proves the proposition.
\end{proof}
\begin{proof}[Proof of Corollary \ref{th:uniformsimpleBounds}]

In the case $\rho_c < 1$, using Proposition \ref{th:uniformResult}
and the fact that $\exp(z) \geq 1+z$, we get $\gamma_2 \leq \rho_c \gamma_2 + \frac {\rho_c n_0}{n-n_0}$ which rewrites $\gamma_2 \leq \frac{\rho_c n_0}{(1-\rho_c)(n-n_0)}$ in the case $\rho_c < 1$. Therefore,
\begin{eqnarray}
\sigma_{\mbox{uniform}} \leq n_0 \left(1+\frac{\rho_c}{1-\rho_c}\right)=\frac{n_0}{1-\rho_c}
\end{eqnarray}
The second claim is straightforward from Proposition \ref{th:uniformResult}
, using the fact that $\gamma_2 \leq 1$.
\end{proof}

\subsection*{Proofs of Lemma \ref{th:lemmaMassoulie}, Lemma \ref{th:BondPerco}, Proposition \ref{th:sizecomponent} and Corollary \ref{th:limsup}}
\begin{proof}[Proof of Lemma \ref{th:lemmaMassoulie}]
Because matrices $\frac{\HazMat(A) + \HazMat(A)^\top}{2}$ and $\frac{\beta}{\delta} \mathcal{A}$ are symmetric and verify $0 \leq \frac{\HazMat(A) + \HazMat(A)^\top}{2} \leq \frac{\beta}{\delta} \mathcal{A} = \mathcal{H}$ where $\leq$ stands for the coefficient-wise inequality, we have $\rho(\frac{\HazMat(A) + \HazMat(A)^\top}{2}) \leq \frac{\beta}{\delta} \rho(\mathcal{A})$ as a direct consequence of the Perron-Frobenius theorem (see e.g \cite{meyer2000matrix}). We now introduce the function $$f: \rho \rightarrow n_0+\sqrt{\frac{\rho}{1-\rho}} \sqrt{n_0(n-n_0)} - \frac{\sqrt{n n_0}}{1- \rho}$$ We have $f(0)<0$ and $f'(\rho)=\sqrt{n_0(n-n_0)} \frac{\rho}{(1-\rho)^{3/2}}-\sqrt{n_0 n} \frac{1}{(1-\rho)^{2}}<0$. Therefore, $f(\rho)<0$ for any $\rho \in [0,1]$, which proves the Lemma.
\end{proof}

\begin{proof}[Proof of Lemma \ref{th:BondPerco}]
First, note that, for bond percolation, the random variables $1_{\{(i,j) \in \mathcal{E}'\}}$ are independent Bernoulli variables $F_{(i,j)} \sim \mathcal{B}(p_{ij})$. We therefore have, similarly than in the proof of Lemma \ref{th:equivalence}
\begin{equation}
\label{eqn:F}
X_i = 1 - \prod_j (1 - X_j F_{(i,j)})
\end{equation}
where $X_i$ is $1$ if node $i$ belongs to the connected component containing the influencer node $v$, and is $0$ otherwise.
We then show that, because $\mathcal{P}$ is symmetric, for any infection process $DTIC(\mathcal{P})$, we can also define independent variables $F'_{(i,j)} \sim \mathcal{B}(p_{ij})$ such that, 
\begin{equation}
\label{eqn:Fprime}
X_i = 1 - \prod_j (1 - X_j F'_{(i,j)})
\end{equation}
Indeed, the event that node $i$ makes an attempt to infect node $j$ will never occur in the same epidemic than  the event that node $j$ makes an attempt to infect node $i$. Therefore, drawing two variables $E_{ij}$ and $E_{ji}$ at the beginning of each epidemic and letting the dynamic decide which of the two results will be used, or drawing only one variable $F'_{(i,j)}\sim \mathcal{B}(p_{ij})$ and using it for each epidemic to decide wether the infection can spread along the edge $(i,j)$ or not is strictly equivalent, given that $E_{ij}$ and $E_{ji}$ are independent and have the same distribution. From equations \ref{eqn:F} and \ref{eqn:Fprime}, we see that, for any $i \in \mathcal{V}$, the probability that a node $i$ is infected is the same for the two processes.
\end{proof}
\begin{proof}[Proof of Proposition \ref{th:sizecomponent}]
By proposition \ref{th:uniformResult}
applied to the case $n_0 = 1$ with the notation $\gamma_3=\frac{(n-1) \gamma_2 +1}{n}$, we get $\sigma_{\mbox{uniform}} \leq n \gamma_3$. We then use the fact that, when the influencer node is uniformly randomly drawn on $\mathcal{V}$, it belongs to the largest connected component and therefore creates an infection of $C_1(\mathcal{G}')$ nodes with probability $\frac{C_1(\mathcal{G}')}{n}$. Therefore, $\Exp{\frac{C_1(\mathcal{G}')}{n} C_1(\mathcal{G}')} \leq  \sigma_{\mbox{uniform}} \leq n \gamma_3$. But $\Exp{C_1(\mathcal{G}')^{2}} \geq \Exp{C_1(\mathcal{G}')}^{2}$ which yields $\Exp{C_1(\mathcal{G}')} \leq n \sqrt{\gamma_3}$.
Moreover, denoting as $C_A(\mathcal{G}')$ the size of the connected component containing the influencer node, we have 
$\sigma_{\mbox{uniform}} = \Exp{C_A(\mathcal{G}')} = \sum_i {i \Prob{C_A(\mathcal{G}') = i}} \geq n \Prob{C_A(\mathcal{G}') = n}=n \Prob{\mathcal{G}' \mbox{ is connected}}$, and therefore $\Prob{\mathcal{G}' \mbox{ is connected}} \leq \gamma_3$.
\end{proof}

\begin{proof}[Proof of Corollary \ref{th:limsup}]
According to Eq.~\ref{eqn:limsupassumption}
, there exists $m \in \mathbb{N}$ and $\eta < 1$ such that for any $n \geq m$, $\rho_c (\mathcal{H}^n)\leq \eta$. Therefore, Corrolary \ref{th:simpleSize}
 implies $\Exp{C_1(\mathcal{G'}_n)} \leq \sqrt{\frac{n}{1-\eta}}$. But for any $\delta > 0$, $\Prob{C_1(\mathcal{G'}_n)>\delta n^{1/2 + \epsilon}} \leq \frac{\Exp{C_1(\mathcal{G'}_n)}} {\delta n^{1/2 + \epsilon}} = o(1)$ which proves the corollary.
\end{proof}

\end{document}